\documentclass[11pt,a4paper,english]{article}
\usepackage[T1]{fontenc}
\usepackage[latin9]{inputenc}
\usepackage{enumitem}
\usepackage{amsmath}
\usepackage{amsthm}
\usepackage{amssymb}

\makeatletter

\theoremstyle{plain}
\newtheorem{thm}{\protect\theoremname}[section]
\theoremstyle{definition}
\newtheorem{defn}[thm]{\protect\definitionname}
\theoremstyle{plain}
\newtheorem{prop}[thm]{\protect\propositionname}
\theoremstyle{plain}
\newtheorem{lem}[thm]{\protect\lemmaname}
\ifx\proof\undefined
\newenvironment{proof}[1][\protect\proofname]{\par
	\normalfont\topsep6\p@\@plus6\p@\relax
	\trivlist
	\itemindent\parindent
	\item[\hskip\labelsep\scshape #1]\ignorespaces
}{%
	\endtrivlist\@endpefalse
}
\providecommand{\proofname}{Proof}
\fi
\theoremstyle{remark}
\newtheorem{rem}[thm]{\protect\remarkname}

\usepackage{authblk}
\date{}

\makeatother

\usepackage{babel}
\providecommand{\definitionname}{Definition}
\providecommand{\lemmaname}{Lemma}
\providecommand{\propositionname}{Proposition}
\providecommand{\remarkname}{Remark}
\providecommand{\theoremname}{Theorem}

\begin{document}
\title{On the dimension of $\alpha\beta$-sets}
\author{Michael Hochman\thanks{Research supported by ISF research grant 3056/21.}}
\maketitle
\begin{abstract}
We show that the Feng-Xiong lower bound of $1/2$ for the box dimension
of $\alpha\beta$-sets is tight. We also study how much of an $\alpha\beta$-orbit
``carries the dimension'': deleting an arbitararily small positive
density set of times can cause the box dimension to drop to zero,
but the Assouad dimension cannot drop below $1/4$.

\tableofcontents{}
\end{abstract}

\section{\label{sec:Introduction}Introduction}

\subsection{Background and results}

Let $\alpha,\beta\in\mathbb{R}$. An $\alpha\beta$-set is a subset
$E\subseteq\mathbb{R}/\mathbb{Z}$ such that for every $x\in E$,
at least one of the points $x+\alpha$, $x+\beta$ belongs to $E$.
The simplest, and in some sense minimal, example is that of an $\alpha\beta$-orbit:
a sequence $(t_{n})_{n=1}^{\infty}\subseteq\mathbb{R}/\mathbb{Z}$
such that $t_{n+1}-t_{n}\in\{\alpha,\beta\}$. One can think of such
an orbit as being generated by two rotations of $\mathbb{R}/\mathbb{Z}$,
applied repeatedly and in no particular order. This question bears
some relation to the multilinear Diophantine properties of $\alpha,\beta$,
because, $k_{n},m_{n}$ are the number of steps up to time $n$ that
are equal to $\alpha,\beta$, respectively, then $t_{n}=k_{n}\alpha+m_{n}\beta$. 

Engelking seems to have been the first to consider the size of $\alpha\beta$-sets,
and he asked whether, assuming that $\alpha,\beta,1$ are rationally
independent, such a set must be somewhere dense \cite{Engelking1961}.
In 1979, Katznelson answered this in the negative: He showed that
every $\alpha\beta$ admit a nowhere-dense $\alpha\beta$-set, and
moreover, he constructed $\alpha,\beta$ and an $\alpha\beta$-set
$E$ of Hausdorff dimension zero \cite{Katznelson1979}. 

Over the past decade the problem has attracted renewed interest due
to its connection intersection and embedding problems of self-similar
sets (see Section \ref{subsec:Aside:-Connection-with-fractal-geometry}
for a sketch of this connection). In their study of embedding problems,
Feng and Xiong \cite{FengXiong2018} showed that while the Hausdorff
dimension of an $\alpha\beta$-set may vanish, when $1,\alpha,\beta$
rationally independent, the lower box dimension is always at least
$1/2$. 

More recently, H. Yu showed that if in addition $\alpha,\beta$
are algebraic then the box dimension is equal to $1$, \cite{Yu2019},
and S. Baker showed showed this to be the case when $\alpha$ is Diophantine
and $\beta$ is Liouvillian \cite{Baker2021}. These results have
lead to the expectation that the box dimension should always be one
(in fact this appeared explicitly as a conjecture in Yu's paper).

The first result of this note is that this is not the case. In fact
we show that the Feng-Xiong bound is tight.\footnote{Feng-Xiong and Yu also considered sets $E\subseteq\mathbb{R}/\mathbb{Z}$
such that $E\subseteq\bigcup_{i=1}^{k}(E-\alpha_{i})$ for a given
set $\{\alpha_{1},\ldots,\alpha_{k}\}$, and under suitable independence
assumptions obtained bounds of the form $\overline{\dim}_{B}(E)\geq1/(k+1)$,
and $>1/k$ when $\alpha_{i}$ are algebraic. It is likely that a
construction similar to ours will show that the $1/k$ bound is tight.} Let $\dim_{B}$ denote box dimension and $\underline{\dim}_{B},\overline{\dim}_{B}$
the lower and upper variants.
\begin{thm}
\label{thm:box-dim-half}There exist $\alpha,\beta\in\mathbb{R}$
with $1,\alpha,\beta$ rationally independen, and an $\alpha\beta$-set
$E$, such that $\overline{\dim}_{B}E=1/2$.
\end{thm}
As it turns out, Katznelson's original construction can be adapted
to give such an example, provided the parameters are carefully chosen.
Unfortunately, Katznelson's paper is rather brief and he only verifies
that the construction works for parameters ``large enough''. A large
part of our work is devoted to a careful analysis of the construction
that shows that it works for the parameters that we need.

The second part of this paper deals with a ``density version'' of
the $\alpha\beta$-set problem. It is motivated by new developments
on the embedding and intersections problems mentioned above, where
it was recently shown that the embeddings conjecture would follow
if the dimension of $\alpha\beta$-orbits were not carried on too
``thin'' sets of times \cite{AlgomHochmanWu2024}. To make things
precise, for $J\subseteq\mathbb{N}$ define the density 
\[
d(J)=\lim_{N\rightarrow\infty}\frac{1}{N}|J\cap\{1,\ldots,N\}|
\]
if it exists. If not, we define the upper and lower densities, $\overline{d}(J),\underline{d}(J)$,
using the limsup and liminf, respectively. We define the following
``Robustness'' property for $\alpha\beta$-orbits: 
\begin{defn}
We say that $\alpha,\beta\in\mathbb{R}$ satisfy Property (R) if for
every $\alpha\beta$-orbit $(t_{n})$ and every $J\subseteq\mathbb{N}$
satisfying $\overline{d}(J)<1$, we have $\dim_{B}\{t_{n}\mid n\in\mathbb{N}\setminus J\}>0$.
\end{defn}
It was shown in \cite{AlgomHochmanWu2024} that the embedding conjecture
would follow if Property (R) holds whenever $1,\alpha,\beta$ are
rationally independent. This is plausible because, Feng and Xiong
showed that under this hypothesis the full $\alpha\beta$-orbit has
positive box dimension \cite{FengXiong2018}, and by a compactness
argument this yields
\begin{prop}
\label{prop:full-density-has-box-dim--half}If $1,\alpha,\beta$ are
rationally independent, $X=\{t_{n}\}$ is an $\alpha\beta$-orbit,
and $J\subseteq\mathbb{N}$ satisfies $\underline{d}(J)=0$, then
$\dim_{B}\{x_{n}\mid n\in\mathbb{N}\setminus J\}\geq1/2$,
\end{prop}
Unfortunately, in general Property (R) fails:
\begin{thm}
\label{thm:full-density-box-dim-zero}There exist $\alpha,\beta\in\mathbb{R}$
with $1,\alpha,\beta$ rationally independent, and an infinite $\alpha\beta$-orbit
$(t_{n})_{n=1}^{\infty}$, such that there exist $J\subseteq\mathbb{N}$
of arbitrarily small lower density satisfying $\dim_{B}\{x_{n}\mid n\in\mathbb{N}\setminus J\}=0$
.
\end{thm}
Nevertheless, for a weaker notion of dimension we are able to prove
a positive result. Let $\dim_{A}E$ denote the Assouad dimension of
a set $E$ (see Section \ref{subsec:Proof-of-Assouad-result} for
the definition).
\begin{thm}
\label{thm:Assouad-along-subsequence}Let $1,\alpha,\beta$ be rationally
independent, let $(t_{n})_{n=1}^{\infty}$ be an $\alpha\beta$-orbit,
and let $J\subseteq\mathbb{N}$ be a set with $\underline{d}(J)<1$.
Then $\dim_{A}\{t_{n}\mid n\in\mathbb{N}\setminus J\}\geq1/4$.
\end{thm}
We have not tried to optimize the value $1/4$.

Theorem \ref{thm:Assouad-along-subsequence} is not currently sufficient
for the applicationss in \cite{AlgomHochmanWu2024}, and leaves the
embedding conjecture unresolved. One may hope that further progress
on the structure of $\alpha\beta$-sets may shed light on that matter.
It is possible, in principle, that every $\alpha\beta$-orbit either
has box dimension one, or that every positive-density subsequence
of it has positive box dimension. Our counter-examples to the two
behaviors are quite different and it is not obvious how to combine
them. But it seems likely that examples combining these behaviors
exist.

\subsection{\label{subsec:Aside:-Connection-with-fractal-geometry}Intersections
of self-similar sets}

Let us briefly explain how $\alpha\beta$-orbits arise in fractal
geometry. Consider a Cantor set $X\subseteq[0,1]$ constructed by
replacing $[0,1]$ by two sub-intervals of length $\lambda$ abutting
on $0,1$, and repeating recursively. Let $Y\subseteq[0,1]$ be constructed
in the same way from a parameter $\gamma$. In the 1970's, Furstenberg
conjectured that when $\log\lambda/\log\gamma$ is irrational, $X,Y$
satisfy certain ``transversality'' properties, and in particular
that for all $a,b\in\mathbb{R}$ with $a\neq0$, 
\[
\dim((aX+b)\cap Y)\leq\max\{0,\dim X+\dim Y-1\}
\]
This was verified several years ago by Shmerkin \cite{Shmerkin2019}
and Wu \cite{Wu2019}. 

The following approach to the problem essentially goes back to Furstenberg.
Let us say that a real number $a$ is ``good'' if there exists $b\in\mathbb{R}$
with $\dim(aX+b)\cap Y=\delta$. Let $G\subseteq(0,\infty)$ be the
set of good parameters. Then 
\begin{enumerate}
\item Since $X$ is the union of two copies of $X$ scaled by $\lambda$,
one of these copies must intersect $Y$ in a set of dimension at least
$\delta$. In other words, there exists $b'$ such that $\dim(\lambda aX+b')\cap Y=\delta$.
This shows that $G$ closed to the map $a\mapsto\lambda a$.
\item Since $Y$ is made of two copies of $Y$ scaled by $\gamma$, one
of these copies intersects $aX$ in a set of dimension $\delta$,
and we can map this intersection back into $Y$ by an affine map with
expansion $1/\gamma$. Thus, there exists $b''$ such that which $\dim((a/\gamma)X+b'')=\delta$.
This shows that $G$ is closed to the map $a\mapsto a/\gamma$.
\end{enumerate}
Taking logarithms and dividing by $\log\lambda$, we find that the
set $H=\log G/\log\lambda$ is closed under translation by $-\log\lambda/\log\gamma$,
and by $1$. Applying these in the appropriate order, we find that
$\log H\cap[0,1)$ is non-empty and invariant under translation by
$-\log\lambda/\log\gamma$ modulo one. The fact that $\log\lambda/\log\gamma\notin\mathbb{Q}$
and $H$ is closed means that $H\cap[0,1)=[0,1)$. This irrational
rotation, and resulting largeness of $G$, plays a role in all existing
work on this problem.

Now, suppose that instead of the construction above, $Y$ is defined
using two parameters $\gamma_{1},\gamma_{2}$, so at each stage of
the construction we replace intervals by two sub-intervals of relative
lengths $\gamma_{1}$ and $\gamma_{2}$. Such a $Y$ is called a non-homogeneous
self-similar set. The same argument in (1) applies and we find that
$G$ is still invariant under $a\mapsto\lambda a$. However, reviewing
(2), we see that $(aX+b)\cap Y$ could intersect either the copy of
$Y$ scaled by $\gamma_{1}$, or the on scaled by $\gamma_{2}$. Consequently,
for every $a\in G$, one of $a/\gamma_{1},a/\gamma_{2}$ is in $G$,
but we do not control which one. It follows that rather than being
invariant under a rotation modulo one, $H\cap[0,1)$ is an $\alpha\beta$-set
for $\alpha=-\log\gamma_{1}/\log\lambda$ and $\beta=-\log\gamma_{2}/\log\lambda$.

These $\alpha\beta$-sets are a key ingredient in all work on intersections
and embeddings of non-homogenerous self-similar. As explained earlier,
Theorems \ref{thm:box-dim-half} and \ref{thm:full-density-box-dim-zero}
indicate some limitations of this approach, although Theorem \ref{thm:Assouad-along-subsequence}
or other properties of $\alpha\beta$-orbits may yet offer an alternative.

\subsection{Organization}

Section \ref{sec:Notation} contains some notation. Then we prove
Theorem \ref{thm:box-dim-half} in Section \ref{sec:An--set-of-Bdim-half},
Proposition \ref{prop:full-density-has-box-dim--half} in Section
\ref{subsec:deleting-zero-density}, Theorem \ref{thm:full-density-box-dim-zero}
in Section \ref{sec:Zero-Bdim-after-positive-density-deletion}, and
Theorem \ref{thm:Assouad-along-subsequence} in Section \ref{sec:Positive-Assouad-dimension-after-positive-density-eletion}. 

\section{\label{sec:Notation}Notation}

The $\rho$-covering number $N(E,\rho)$ of $E$ is the minimal number
of $\rho$-intervals needed to cover $E\subseteq\mathbb{R}$. The
upper box dimension of $E$ is defined by
\[
\overline{\dim}_{B}E=\limsup_{\rho\searrow0}\frac{\log N(E,\rho)}{\log(1/\rho)}
\]
the lower box dimension $\underline{\dim}_{B}E$is defined similarly
with $\liminf$ and if the limit exists it is denoted $\dim_{B}E$.
Note that these quantities are the same for $E$ and its closure.

For $\alpha,\beta\in\mathbb{R}$, an $\alpha\beta$-orbit is a finite
or infinite sequence $(t_{n})\subseteq\mathbb{R}/\mathbb{Z}$ with
gaps $t_{n+1}-t_{n}\in\{\alpha,\beta\}$. An $\alpha\beta$-orbit
starting at $0$ is uniquely determined by these gaps, and, conversely,
any sequence of gaps $w_{n}\in\{\alpha,\beta\}$ defines a unique
$\alpha\beta$-orbit, starting from $0$. 

Let $\{x,y\}^{*}$ denote the set of finite sequences in the formal
symbols $x,y$. We write $|W|_{x},|W|_{y}$ for the number of times
the symbols $x,y$ appear in $W$ respectively.

A word $W\in\{x,y\}^{*}$ can be thought of as a template an $\alpha\beta$-orbit:
Given $\alpha,\beta\in\mathbb{R}$, upon setting $x=\alpha$, $y=\beta$,
we get a sequence of gaps of size $\alpha,\beta$, and an associated
orbit started from $0$, which we denote by $W[\alpha,\beta]\subseteq\mathbb{R}/\mathbb{Z}$,
and call the $\alpha\beta$-orbit of $W$. When $W$ is finite, $W(\alpha,\beta)$
denotes the terminal elements of $W[\alpha,\beta]$, and we say that
the orbit is closed if $W(\alpha,\beta)=0$. 

For $J\subseteq\mathbb{N}$ we write $W[\alpha,\beta]|_{J}$ for the
restriction of the $\alpha\beta$-orbit $W[\alpha,\beta]$ to the
times in $J$.

Our convention will be that all parameters $\alpha,\beta$ etc.~are
real (or natural) numbers, but all $\alpha\beta$-orbits lie in $\mathbb{R}/\mathbb{Z}$.
This is already evident in the definitions above, where $\alpha,\beta\in\mathbb{R}$
and $W[\alpha,\beta]\subseteq\mathbb{R}/\mathbb{Z}$. When necessary,
we emphasize that certain quantities are real, but most of the time
we live with the ambiguity.

\section{\label{sec:An--set-of-Bdim-half}An $\alpha\beta$-set of box dimension
$1/2$ }

\subsection{Sketch of the construction}

The general idea, adjusted to our objective, is as follows. One begins
with (rational) parameters $\alpha,\beta>0$ satisfying $\beta=N\alpha$
and $\alpha+N\alpha+M\beta=\alpha+(M+1)\beta=1$ for some $M,N\in\mathbb{N}$.
These define the $\alpha\beta$-set 
\[
E=\{0,\alpha,2\alpha,\ldots,N\alpha,(N+1)\alpha=\alpha+\beta,\alpha+2\beta,\ldots\alpha+(M+1)\beta=1\}
\]
The covering number of $E$ at scale $1/MN$ is $|E|=M+N-1$, so if
we choose $M=N$, or even $N=M^{1+\varepsilon}$, the set $E$ has
``box dimension $\approx1/2$ at scale $1/MN$''. But of course,
these $\alpha,\beta$ are rational, $E$ is finite, so in reality
 $\overline{\dim}_{B}(E)=0$. To continue, perturb $\alpha,\beta$
to $\alpha',\beta'$ so as to get a longer (but still finite) $\alpha'\beta'$-orbit
that remains close to the previous one, thus preserving the behavior
at the original scale $1/MN$, and such that the deviation of the
new orbit from the original one looks like a scaled-down copy of the
set $E$ above, so the new orbit looks ``$1/2$-dimensional'' both
at scale $1/MN$ and at another smaller scale. Iterating this procedure
gives an infinite $\alpha\beta$-set of dimension $1/2$ in the limit.

In this construction the choice of $M,N$ and their relation to subsequent
parameters is crucial for the dimension bound. Unfortunately, Katznelson's
paper does not provide much detail on what the permissible parameters
are, only taking them ``sufficiently large''. This is not an option
for us: indeed, if $N\gg M$ then the set $E$ above has ``box dimension
$\approx1$ at scale $1/MN$''. 

\subsection{Katznelson's construction}

We are going to recursively define $\alpha_{n},\beta_{n}\in\mathbb{Q}$
and words $U_{n},V_{n}\in\{x,y\}^{*}$ defining closed $\alpha_{n},\beta_{n}$-orbits,
so that the set $E_{n}=U_{n}[\alpha_{n},\beta_{n}]\cup V_{n}[\alpha_{n}\cup\beta_{n}]$,
consisting of all points in these orbits, is an $\alpha_{n}\beta_{n}$-set.
We will eventually take $\alpha=\lim\alpha_{n}$, $\beta=\lim\beta_{n}$
and $E=\lim E_{n}$ (in a suitable sense), and show that $E$ is an
$\alpha\beta$-set with the desired properties. 

The input to the constructions is a sequence of parameters $M_{1},N_{1},M_{2},N_{2},\ldots\in\mathbb{N}$.
One should think of them as growing fairly rapidly; we shall be more
precise later.

To begin, choose $\alpha_{1},\beta_{1}$ such that we can go from
$0$ to $1$ in the real line (and also in $\mathbb{R}/\mathbb{Z}$)
by taking $(N_{1}+1)$ $\alpha_{1}$-steps followed by $M_{1}$ $\beta_{1}$-steps,
and also by taking a single $\alpha_{1}$-step followed by $(M_{1}+1)$
$\beta_{1}$-steps. These $\alpha_{1}\beta_{1}$-orbits are described
by the words
\begin{align*}
U_{1} & =x^{N_{1}+1}y^{M_{1}}\\
V_{1} & =xy^{M_{1}+1}
\end{align*}
and the fact that both orbits are closed means that
\begin{align*}
(N_{1}+1)\alpha_{1}+M_{1}\beta_{1} & =1\\
\alpha_{1}+(M_{1}+1)\beta_{1} & =1
\end{align*}
These identities imply a third one, $\beta_{1}=N_{1}\alpha_{1}$,
so we may imagine that $\alpha_{1}\ll\beta_{1}$.

For the induction step, assume that for $1\leq k<n$ we have constructed
$\alpha_{k},\beta_{k}\in\mathbb{R}$ and $U_{k},V_{k}\in\{x,y\}^{*}$
describing closed $\alpha_{k}\beta_{k}$-orbits. Assume $U_{n-2}$
is a prefix of $U_{n-1},V_{n-1}$ (with the convention $U_{0}=x$,
this holds also for $n=2$), and let $W_{n-1}$ denote the part of
$V_{n-1}$ that comes after $U_{n-2}$, so $V_{n-1}=U_{n-2}W_{n-1}.$
Finally, define the real number 
\[
\varepsilon_{n-1}=\text{the lift of }U_{n-2}(\alpha_{n-1},\beta_{n-1})\text{ to }[0,1)
\]
and assume that $\varepsilon_{n-1}>0$ is the closest point to the
right of $0$ in both of the $\alpha_{n-1},\beta_{n-1}$-orbits $U_{n-1},V_{n-1}$.
Note that, starting from $\varepsilon_{n-1}$, the $\alpha_{n-1}\beta_{n-1}$-orbit
$W_{n-1}$ takes us from $\varepsilon_{n-1}$ back to $0\bmod1$ . 

We define $\alpha_{n},\beta_{n}\in\mathbb{R}$ by slightly perturbing
$\alpha_{n-1},\beta_{n-1}$. As a result, the $\alpha_{n}\beta_{n}$-orbits
$U_{n-1},V_{n-1}$ will overshoot $0$ by some $\varepsilon_{n},\eta_{n}>0$,
respectively, and if the perturbation is small we will have $0<\varepsilon_{n}\ll\eta_{n}\ll\varepsilon_{n-1}$.
Our plan is to use steps of size $\varepsilon_{n},\eta_{n}$ to span
the interval $[0,\varepsilon_{n-1}]$ in two different ways, as we
did in the first step with $\alpha_{1},\beta_{1}$. Specifically,
we span $[0,1]$ by taking $(N_{n}+1)$ $\varepsilon_{n}$-steps followed
by $M_{n}$ $\eta_{n}$-steps, or by taking one $\varepsilon_{n}$
step followed by $(N_{n}+1$) $\eta_{n}$-steps. Each of these $\varepsilon_{n}\eta_{n}$-orbits
can be realized as an $\alpha_{n}\beta_{n}$-orbit by repeating the
$\alpha_{n}\beta_{n}$-orbits of $U_{n-1},V_{n-1}$ the prescribed
number of times, and these $\alpha_{n}\beta_{n}$-orbits can be closed
by attaching the $\alpha_{n}\beta_{n}$-orbit $W_{n-1}$ at the end,
taking us from $\varepsilon_{n-1}$ back to $0$.

There is a small flaw with this plan: when we perturb $\alpha_{n-1},\beta_{n-1}$
to $\alpha_{n},\beta_{n}$, the path $W_{n-1}$ also changes, by some
small $\delta_{n}\in\mathbb{R}$, and so it no longer connects $\varepsilon_{n-1}$
to $0$, but rather connects to $\varepsilon_{n-1}-\delta$, and we
must take this into account when choosing the perturbation. 

Summarizing, the new paths that we define are given by the words
\begin{align}
U_{n} & =(U_{n-1})^{N_{n}}V_{n}^{M_{n}}W_{n-1}\nonumber \\
V_{n} & =U_{n-1}V_{n-1}^{M_{n}+1}W_{n-1}\label{eq:definition-of-UnVn}
\end{align}
We perturb $\alpha_{n-1},\beta_{n-1}$ to $\alpha_{n},\beta_{n}$,
with resulting changes of $\varepsilon_{n},\eta_{n},\delta_{n}\in\mathbb{R}$
to the orbits $U_{n-1},V_{n-1},W_{n-1}$, and choose the perturbation
to satisfy the following identities over $\mathbb{R}$: 
\begin{align}
(N_{n}+1)\varepsilon_{n}+M_{n}\eta_{n} & =\varepsilon_{n-1}-\delta_{n}\nonumber \\
\varepsilon_{n}+(M_{n}+1)\eta_{n} & =\varepsilon_{n-1}-\delta_{n}\label{eq:epsilon-eta-recursion}
\end{align}
As a by-product we also have the identity $\eta_{n}=(M_{n}-1)\varepsilon_{n}$.

With regard to the sequences $W_{n}$, we have $W_{1}=y^{M_{1}+1}$
(because $U_{0}=x$ and $V_{1}=xy^{M_{1}+1}$). For $n\geq2$, since
the closest point to $0$ in the orbit $U_{n}[\alpha_{n},\beta_{n}]$
will be $\varepsilon_{n}=U_{n-1}(\alpha_{n},\beta_{n})$, and since
$V_{n}=U_{n-1}V_{n-1}^{M_{n}+1}W_{n-1}$, it follows that 
\begin{equation}
W_{n}=V_{n-1}^{M_{n}+1}W_{n-1}\label{eq:definition-of-Wn}
\end{equation}

\subsection{Analysis of the construction}

\subsubsection{Growth condition}

We assume that the parameters $M_{1}<N_{1}<M_{2}<N_{2}<\ldots$ obey
the following condition for some $0<\gamma<1$ (the exact value of
$\gamma$ will not matter), with the convention $N_{0}=1$:
\begin{description}
\item [{Growth~condition:}] $\sum_{n=1}^{\infty}M_{n}/N_{n}<$$\gamma$
and $\sum_{n=1}^{\infty}N_{n-1}/M_{n}<\gamma$.
\end{description}
This implies monotonicity: $M_{1}<N_{1}<M_{2}<N_{2}<\ldots$, and
also gives a lower bound on the parameters: $M_{n},N_{n}>1/\gamma$. 

\subsubsection{Word lengths}

Let $|W|$ denote the length of a word $W$. The definition of $U_{n},V_{n},W_{n}$
implies some elementary relations between their lengths, which we
use liberally: Each of these sequences increases in length, $|W_{n}|\leq|V_{n}|$,
and a simple induction shows that $|V_{n}|<|U_{n}|$. From this and
the recursive definition we get
\begin{align}
(M_{n}+2)|V_{n-1}|< & |V_{n}|<(M_{n}+2)|U_{n-1}|\nonumber \\
(N_{n}+M_{n}+1)|V_{n-1}|< & |U_{n}|<(N_{n}+M_{n}+2)|U_{n-1}|\label{eq:basic-length-properties}
\end{align}
Also we clearly have $|U_{n-1}|<|V_{n}|$, which, combined with the
previous inequality, gives
\begin{equation}
|U_{n}|<(N_{n}+M_{n}+2)|V_{n}|\label{eq:Un-Vn-comparison}
\end{equation}

\subsubsection{Letter frequencies and $\alpha\beta$-orbits}

Recall that $|W|_{x},|W|_{y}$ denote the number of times $x,y$ appear
in $W$. Note that 
\[
W(\alpha,\beta)=\alpha\cdot|W|_{x}+\beta\cdot|W|_{y}\bmod1
\]
in particular,
\begin{equation}
W(\alpha+s,\beta+t)-W(\alpha,\beta)=s\cdot|W|_{x}+t\cdot|W|_{y}\label{eq:length-of-orbit}
\end{equation}
Let $\mu_{W}=(|W|_{x},|W|_{y})$ be the vector counting occurrences
of $x,y$ in $W$, which we think of as a measure on $\{x,y\}$, and
let $\overline{\mu}_{W}=\frac{1}{|W|}\mu_{W}$ denote the normalized
version, which gives the symbol frequencies. 

\subsubsection{Evolution of $\alpha_{n},\beta_{n}$}

We return to the recursive definition of $\alpha_{n},\beta_{n}$.
Introduce real variables $s,t$ representing the perturbation change
to $\alpha_{n-1},\nu_{n-1}$: 
\begin{align*}
\alpha_{n} & =\alpha_{n-1}+s\\
\beta_{n} & =\beta_{n-1}+t
\end{align*}
Recall that $\varepsilon_{n-1}\in(0,1)$ is the $\mathbb{R}$-lift
of. We define $\varepsilon_{n},\eta_{n},\delta_{n}\in\mathbb{R}$
to be the $\mathbb{R}$-lifts of the following functions of $s,t$,
normalized so that they are zero when $s=t=0$:
\begin{align*}
\varepsilon_{n} & =\varepsilon_{n}(s,t)=U_{n-1}(\alpha_{n-1}+s,\beta_{n-1}+t)\bmod1\\
\eta_{n} & =\eta_{n}(s,t)=V_{n-1}(\alpha_{n-1}+s,\beta_{n-1}+t)\bmod1\\
\delta_{n} & =\delta_{n}(s,t)=W_{n-1}(\alpha_{n-1}+s,\beta_{n-1}+t)-W_{n-1}(\alpha_{n-1},\beta_{n-1})\bmod1
\end{align*}
Re-write the equations (\ref{eq:epsilon-eta-recursion}), which ensure
that the $\alpha_{n}\beta_{n}$-orbits $U_{n},V_{n}$ terminate at
$0$, as
\begin{align}
\varepsilon_{n} & +c_{n}\delta_{n}=c_{n}\varepsilon_{n-1}\qquad\text{with }c_{n}=\frac{1}{M_{n}(1+N_{n})+1}\nonumber \\
\eta_{n} & +d_{n}\delta_{n}=d_{n}\varepsilon_{n-1}\qquad\text{with }d_{n}=\frac{N_{n}}{N_{n}(1+M_{n})+1}\label{eq:main-equation-elementary-form}
\end{align}

If we express $\varepsilon_{n},\eta_{n},\delta_{n}$ using the frequencies
of letters in $U_{n-1},V_{n-1},W_{n-1}$, as in (\ref{eq:length-of-orbit}),
then the last equations can be re-written as 
\begin{equation}
\left(D_{n}A_{n}+D'_{n}A'_{n}\right)\left(\begin{array}{c}
s\\
t
\end{array}\right)=\left(\begin{array}{c}
c_{n}\varepsilon_{n-1}\\
d_{n}\varepsilon_{n-1}
\end{array}\right)\label{eq:main-equation-matix-form}
\end{equation}
where
\[
\begin{array}{cc}
A_{n}=\left(\begin{array}{cc}
\overline{\mu}_{U_{n}}(x) & \overline{\mu}_{U_{n}}(y)\\
\overline{\mu}_{V_{n}}(x) & \overline{\mu}_{V_{n}}(y)
\end{array}\right) & D_{n}=\left(\begin{array}{cc}
|U_{n-1}| & 0\\
0 & |V_{n-1}|
\end{array}\right)\\
A'_{n}=\left(\begin{array}{cc}
\overline{\mu}_{W_{n-1}}(x) & \overline{\mu}_{W_{n-1}}(y)\\
\overline{\mu}_{W_{n-1}}(x) & \overline{\mu}_{W_{n-1}}(y)
\end{array}\right) & D'_{n}=\left(\begin{array}{cc}
c_{n}|W_{n-1}| & 0\\
0 & d_{n}|W_{n-1}|
\end{array}\right)
\end{array}
\]
Multiplying (\ref{eq:main-equation-matix-form}) on the left by $D_{n}^{-1}$,
we get 
\begin{equation}
(A_{n}+D_{n}^{-1}D'_{n}A'_{n})\left(\begin{array}{c}
s\\
t
\end{array}\right)=D_{n}^{-1}\left(\begin{array}{c}
c_{n}\varepsilon_{n-1}\\
d_{n}\varepsilon_{n-1}
\end{array}\right)\label{eq:main-equation-simplified-matrix-form}
\end{equation}

\begin{lem}
\label{lem:letter-frequencies}$A_{n}$ is $O(\gamma)$-close to the
identity matrix, uniformly in $n$. 
\end{lem}
\begin{proof}
By definition,
\[
A_{1}=\left(\begin{array}{cc}
\frac{1}{N_{1}+M_{1}+1}N_{1} & \frac{1}{N_{1}+M_{1}+1}M_{1}\\
\frac{1}{M_{1}+2} & \frac{1}{M_{1}+2}(M_{1}+1)
\end{array}\right)
\]
The growth condition implies $1/M_{1},M_{1}/N_{1}<\gamma$, so $A_{1}=I+O(\gamma)$.
Next, the recursive definition of the words leads to linear recursions
for the letter counts; e.g. 
\[
|U_{n}|_{x}=(N_{n}+1)|U_{n-1}|_{x}+M_{n}|V_{n-1}|_{x}+|W_{n-1}|_{x}
\]
This, in turn, gives linear recurrence relations for $\mu_{U_{n}},\mu_{V_{n}}$
(and $\mu_{W_{n}}$, but we omit it):
\begin{align*}
\mu_{U_{n}} & =N_{n}\mu_{U_{n-1}}+M_{n}\mu_{V_{n-1}}+\mu_{W_{n-1}}\\
\mu_{V_{n}} & =\mu_{U_{n-1}}+(M_{n}+1)\mu_{V_{n-1}}+\mu_{W_{n-1}}
\end{align*}
Dividing the first identity by $|U_{n}|$ we express $\overline{\mu}_{U_{n}}$
as convex combinations of $\overline{\mu}_{U_{n-1}}$ and another
measure, and similarly for $\overline{\mu}_{V_{n}}$. In general if
$\mu=(1-w)\mu'+w\nu$ and $0<w<1/2$, then $\left\Vert \mu-\mu'\right\Vert _{\infty}\leq w/(1-w)\leq2w$.
Using this, $|U_{n}|\geq N_{n}|U_{n-1}|$, $|V_{n}|\geq M_{n}|V_{n-1}|$,
$N_{n}>M_{n}$ and (\ref{eq:Un-Vn-comparison}), we get
\begin{align*}
\left\Vert \overline{\mu}_{U_{n}}-\overline{\mu}_{U_{n-1}}\right\Vert _{\infty} & \leq2\frac{M_{n}|V_{n-1}|+|W_{n-1}|}{N_{n}|U_{n-1}|}<2\frac{(M_{n}+1)|V_{n-1}|}{N_{n}|U_{n-1}|}<4\frac{M_{n}}{N_{n}}\\
\left\Vert \overline{\mu}_{V_{n}}-\overline{\mu}_{V_{n-1}}\right\Vert _{\infty} & \leq2\frac{|U_{n-1}|+|W_{n-1}|}{M_{n}|V_{n-1}|}<4\frac{|U_{n-1}|}{M_{n}|V_{n-1}|}\leq4\frac{N_{n-1}+M_{n-1}+2}{M_{n}}\leq\frac{16N_{n-1}}{M_{n}}
\end{align*}
Summing from $1$ to $n$ gives 
\[
\left\Vert \overline{\mu}_{U_{n}}-\overline{\mu}_{U_{1}}\right\Vert _{\infty}<4\sum_{1<\ell\leq n}\frac{M_{\ell}}{N_{\ell}}<4\gamma\qquad,\qquad\left\Vert \overline{\mu}_{V_{n}}-\overline{\mu}_{V_{1}}\right\Vert _{\infty}<32\sum_{k<\ell\leq n}\frac{N_{\ell-1}}{M_{\ell}}<32\gamma
\]
the claim follows.
\end{proof}
\begin{lem}
\label{lem:bounding-s-t-epsilon-eta-delta}For $\gamma$ small enough,
Equation (\ref{eq:main-equation-simplified-matrix-form}) can be solved
and its solution $(s,t)^{T}$ satisfies
\begin{equation}
|s|,|t|\leq O(\frac{1}{M_{n}|V_{n-1}|}\varepsilon_{n-1})\label{eq:s-t-bounds}
\end{equation}
In addition, $\varepsilon_{n},\eta_{n}>0$ and 
\begin{equation}
\varepsilon_{n}=\frac{1+O(1/M_{n})}{M_{n}N_{n}}\varepsilon_{n-1}\quad,\quad\eta_{n}=\frac{1+O(1/M_{n})}{M_{n}}\varepsilon_{n-1}\quad,\quad|\delta_{n}|=O(\frac{1}{M_{n}}\varepsilon_{n-1})\label{eq:bounding-epsilon-eta-and-delta}
\end{equation}
\end{lem}
\begin{proof}
We first claim that if $\gamma$ is small enough, then $A_{n}-D_{n}^{-1}D'_{n}A'_{n}$
is invertible and both it and its inverse lies in an $O(\gamma)$-neighborhood
of the identity. Indeed, by Lemma \ref{lem:letter-frequencies} the
matrices $A_{n}$ are within $O(\gamma)$ of the identity, and by
(\ref{eq:basic-length-properties}), the coefficients of $D_{n}^{-1}D'_{n}A'_{n}$
are bounded by $\max\{c_{n},d_{n}\}=O(\gamma)$, so $A_{n}-D_{n}^{-1}D'_{n}A'_{n}$
is still within $O(\gamma)$ of the identity, and the claim follows.

Therefore we can apply $(A_{n}-D_{n}^{-1}D'_{n}A'_{n})^{-1}$ to solve\,Equation
(\ref{eq:main-equation-simplified-matrix-form}) and find that, up
to a multiplicative constant, $|s|,|t|$ are bounded by the larger
of $c_{n}\varepsilon_{n-1}/|U_{n-1}|$ and $d_{n}\varepsilon_{n-1}|V_{n-1}|$,
which gives (\ref{eq:s-t-bounds}).

Next, using (\ref{eq:length-of-orbit}) for $W_{n-1}$ with the bounds
on $s,t$ gives 
\[
|\delta_{n}|=O(|W_{n-1}|\cdot\frac{1}{M_{n}|V_{n-1}|}\varepsilon_{n-1})\leq O(\frac{1}{M_{n}}\varepsilon_{n-1})
\]
The stated expression for $\eta_{n}$ follows from the defining relation
$\eta_{n}=d_{n}(\varepsilon_{n-1}+\delta_{n})$ and the expression
for $\varepsilon_{n}$ follows from the relation $\varepsilon_{n}=\eta_{n}/(N_{n}-1)$.
Finally, positivity of $\varepsilon_{n},\eta_{n}$ follow from these
expressions as soon as $1/M_{1}$ (i.e. as soon as $\gamma$) is small
enough.
\end{proof}
\begin{lem}
\label{lem:perturbation-one-step}Assuming $\gamma$ is small enough,
for every $n<\ell$, each point on the $\alpha_{\ell-1}\beta_{\ell-1}$-orbits
of $U_{n},V_{n},W_{n}$ changes by $O(\gamma\varepsilon_{\ell-1})$
when $\alpha_{\ell-1},\beta_{\ell-1}$ are perturbed to $\alpha_{\ell},\beta_{\ell}$.
\end{lem}
\begin{proof}
The points in question correspond to a prefix $Z$ of one of the words
$U_{n},V_{n},W_{n}$. By (\ref{eq:basic-length-properties}) and the
bounds (\ref{eq:s-t-bounds}) for $s,t$, perturbing $\alpha_{\ell-1},\beta_{\ell-1}$
by $s,t$ respectively will move the point by at most $O(|Z|/M_{\ell}|V_{\ell-1}|)\varepsilon_{\ell-1}$.
If $Z$ is a prefix of $V_{n}$ or $W_{n}$ then $|Z|\leq|V_{n}|\leq|V_{\ell-1}|$
and the perturbation is bounded by $O(\varepsilon_{\ell-1}/M_{\ell})$.
If $Z$ is a prefix of $U_{n}$ then $|Z|\leq|U_{n}|\leq|U_{\ell-1}|$,
so by (\ref{eq:Un-Vn-comparison}), $N_{\ell-1}>M_{\ell-1}$and $N_{\ell-1}/M_{\ell}<\gamma$
, the point moves no more than
\[
O(\frac{|U_{\ell-1}|}{|V_{\ell-1}|M_{\ell}}\varepsilon_{\ell-1})\leq O(\frac{M_{\ell-1}+N_{\ell-1}+2}{M_{\ell}}\varepsilon_{\ell-1})=O(\frac{N_{\ell-1}}{M_{\ell}}\varepsilon_{\ell-1})=O(\gamma\varepsilon_{\ell-1})\qedhere
\]
\end{proof}

\subsubsection{Evolution of the sets $E_{n}$}

Recall that $E_{n}$ is the set of points in the $\alpha_{n}\beta_{n}$-orbits
$U_{n},V_{n}$, and observe that this is the same as the set of points
on the $U_{n}$ orbit alone, since, for the parameters $\alpha_{n}b_{n}$,
the $V_{n}$ orbit is a subset of the $U_{n}$ orbit by construction. 

For $k\geq n$, let $E_{n,k}$ denote the set of points in the $\alpha_{k}\beta_{k}$-orbits
$U_{n}$ and $V_{n}$, so $E_{n}=E_{n,n}$ and $E_{n,k}\subseteq E_{k}$
for $k>n$. 

Let $A^{(r)}$ denote the $r$-neighborhood of a set $A\subseteq\mathbb{R}$.
The Hausdorff distance between compact sets $A,B\subseteq[0,1]$ is
\[
d(A,B)=\inf\{r>0\mid A\subseteq B^{(r)}\;,\;B\subseteq A^{(r)}\}
\]
This is a compact metric on the space of compact subsets of $[0,1]$,
and induces a topology characterized by the property that $A_{n}\rightarrow A$
if and only if, for every convergent sequence $a_{1}\in A_{1},a_{1}\in A_{2},\ldots$
we have $\lim a_{n}\in A$.
\begin{lem}
\label{lem:En-properties}Assuming $\gamma$ is small enough,
\begin{enumerate}
\item $d(E_{n},E_{n,k})<0.01\varepsilon_{n}$ for all $k>n\geq n_{0}$.
\item $d(E_{n},E_{n+1})<1.02\varepsilon_{n}$, and $d(E_{n},E_{k})<O(\varepsilon_{n})$
for $k>n$.
\item For every $k>n$ the set $E_{n,k}$ contains a $\frac{1}{2}\varepsilon_{n}$-separated
subset of size at least $\prod_{1\leq\ell\leq n}N_{\ell}$.
\item $\prod_{1\leq\ell\leq n}N_{\ell}\leq|E_{n}|\leq\prod_{1\leq\ell\leq n}(M_{\ell}+N_{\ell}+2)$.
\end{enumerate}
\end{lem}
\begin{proof}
(1) By Lemma (\ref{lem:bounding-s-t-epsilon-eta-delta}) and the growth
assumption, $\sum_{\ell\geq n}\varepsilon_{\ell}=O(\varepsilon_{n})$.
Iterating Lemma \ref{lem:perturbation-one-step}, for $k>n$, 
\[
d(E_{n,n},E_{n,k})\leq\sum_{\ell=n}^{k-1}d(E_{n,\ell},E_{n,\ell+1})\leq\sum_{\ell\geq n}O(\gamma\varepsilon_{\ell})=O(\gamma\varepsilon_{n})
\]
Which is (1) if $\gamma$ is small enough. 

(2) Note that by definition, $E_{n,n+1}\subseteq E_{n+1}$ (because
$U_{n}$ is a prefix of $U_{n+1}$), and $d(E_{n},E_{n,n+1})\leq0.01\varepsilon_{n}$
by (1), so $E_{n}\subseteq E_{n+1}^{(0.01\varepsilon_{n})}$. On the
other hand, $E_{n+1}$ consists of the translates of the $\alpha_{n+1}\beta_{n+1}$
orbits $U_{n},V_{n}$, and hence of $E_{n,n+1}$, by elements of the
set $\{\ell\varepsilon_{n+1}\}_{0\leq\ell<N_{n+1}}\cup\{\ell\eta_{n+1}\}_{0\leq\ell<M_{n+1}}$,
whose diameter is $\varepsilon_{n}+\delta_{n+1}=(1+O(\gamma))\varepsilon_{n}$
(using (\ref{eq:epsilon-eta-recursion})), so
\[
E_{n+1}\subseteq(E_{n,n+1})^{(1+O(\gamma))\varepsilon_{n})}\subseteq(E_{n}^{(0.01\varepsilon_{n})})^{(1+O(\gamma))\varepsilon_{n}}\subseteq E_{n}^{(1.02\varepsilon_{n})}
\]
if $\gamma$ is small enough. We have shown that $d(E_{n},E_{n+1})=1.02\varepsilon_{n}$,
and iterating, for $k>n$ we get $d(E_{n},E_{k})=1.02\sum_{\ell\geq n}\varepsilon_{\ell}=O(\varepsilon_{n})$.

(3) We first claim that each $E_{n}$ contain a $\frac{2}{3}\varepsilon_{n}$-separated
set $F_{n}$. For $n=1$ this is immediate (note that $\varepsilon_{1}=\alpha_{1}$).
Now suppose that $F_{n}\subseteq E_{n}$ is $\frac{2}{3}\varepsilon_{n}$-separated.
By Lemma \ref{lem:perturbation-one-step}, the perturbation $F'_{n}\subseteq E_{n,n+1}$
of $D_{n}$ satisfies $d(F_{n},F'_{n})<O(\gamma\varepsilon_{n})<0.01\varepsilon_{n}$
(assuming $\gamma$ small enough), so $F'_{n}$ is $0.664\varepsilon_{n}$-separated.
Since $E_{n+1}$ contains all the translates $E_{n,n+1}+\ell\varepsilon_{n+1}$
for $0\leq\ell<N_{n+1}$, the set $F_{n+1}=F'_{n}+\{\ell\varepsilon_{n+1}\}_{0\leq\ell<N_{n+1}}$
is contained in $E_{n+1}$; and since the gaps in $F'_{n}$ are of
size $0.664\varepsilon_{n},$which is much longer than the diameter
$N_{n+1}\varepsilon_{n+1}$ of $\{\ell\varepsilon_{n+1}\}_{0\leq\ell<N_{n+1}}$,
clearly $F_{n+1}$ is $\varepsilon_{n+1}$ separated. 

To complete (3) given an $\frac{2}{3}\varepsilon_{n}$-separated set
$F_{n}\subseteq E_{n}$, the perturbations $F_{n,k}\subseteq E_{n,k}$
satisfy $d(F_{n}F_{n,k})<0.01\varepsilon_{n}$, so $F_{n,k}$is $\frac{1}{2}\varepsilon_{n}$-separated,
as required.

The left hand side of (4) follows from (3). For the right hand side,
use $|E_{n}|\leq|U_{n}|$, and iterate the bound $|U_{n}|\leq(N_{n}+M_{n}+2)|U_{n-1}|$.
\end{proof}
\begin{lem}
Assuming $\gamma$ is small enough,
\begin{enumerate}
\item The limits $\alpha=\lim\alpha_{n}$ , $\beta=\lim\beta_{n}$ exist.
\item $(E_{n})$ converges in the Hausdorff metric to a compact $E\subseteq[0,1]$,
and $d(E_{n},E)\leq O(\varepsilon_{n})$. 
\item $E$ is the closure of the infinite $\alpha\beta$-orbit $U[a,\beta]$
defined by the infinite word $U\in\{0,1\}^{\mathbb{N}}$ that is the
common extension of the $U_{n}$.
\item $E$ is infinite and at least one of $\alpha,\beta$ is irrational.
\end{enumerate}
\end{lem}
\begin{proof}
Equation (\ref{eq:s-t-bounds}) implies $\sum|\alpha_{n+1}-\alpha_{n}|<\infty$
and $\sum|\beta_{n+1}-\beta_{n}|<\infty$ , and (1) follows.

By part (2) of Lemma (\ref{lem:En-properties}) $(E_{n})$ is a Cauchy
sequence, and the Hausdorff metric is complete, so $E=\lim E_{n}$
exists. Taking $k\rightarrow\infty$ in the bound $d(E_{n},E_{k})<O(\varepsilon_{n})$
gives $d(E_{n},E)\leq O(\varepsilon_{n})$.

Clearly $E_{n,k}=U_{n}[\alpha_{k},\beta_{k}]\rightarrow U_{n}[\alpha,\beta]$
as $k\rightarrow\infty$, so by part (1) of Lemma \ref{lem:En-properties},
$d(E_{n},U_{n}[\alpha,\beta])\leq0.01\varepsilon_{n}$. Since $E_{n}\rightarrow E$
also $U_{n}[\alpha,\beta]\rightarrow E$. On the other hand, clearly
$U_{n}[\alpha,\beta]\rightarrow\overline{U[\alpha,\beta]}$, so $E=\overline{U[\alpha,\beta]}$,
giving (3).

Finally, $E_{n,k}$ contains an $\frac{1}{2}\varepsilon_{n}$-separated
set of size $\prod_{1\leq\ell\leq n}N_{\ell}$, so the same is true
for $U_{n}[\alpha,\beta]=\lim_{k\rightarrow\infty}E_{n,k}$. Since
$U_{n}[\alpha,\beta]\subseteq U[\alpha,\beta]\subseteq E$ it follows
that $E$ is infinite. If $\alpha,\beta$ we both rational then $U[\alpha,\beta]$,
and hence $E=\overline{U[\alpha,\beta]}$, would be finite. So at
least one of $\alpha,\beta$ is rational.
\end{proof}

\subsection{Choosing parameters and estimating the dimension}

Fix a large $L\in\mathbb{N}$ and apply the construction from the
previous sections for the 
\[
M_{n}=2^{(2(n+L))^{2}}\quad,\quad N_{n}=2^{(2(n+L)+1)^{2}}
\]
Then $M_{n}/N_{n}<2^{-4(n+L)}$ and $N_{n-1}/M_{n}<2^{-4(n-2+L)}$,
so the growth condition is satisfied with $\gamma=2^{-4L}$, which
we may assume is small enough for all the previous lemmas to hold.
By (\ref{eq:bounding-epsilon-eta-and-delta}), the growth condition
(which ensures that $\sum1/M_{n}<\infty$), and the definition of
$M_{n},N_{n}$,
\begin{equation}
\varepsilon_{n}=\Theta(\frac{1}{M_{1}\ldots M_{n}N_{1}\ldots N_{n}})=2^{(4+o(1))n^{3}}\label{eq:epsilon-n-bound}
\end{equation}

It remains to estimate the dimension of $E$. Recall that $N(X,\rho)$
denotes the $\rho$-covering number $X$ and 
\[
\overline{\dim}_{B}E=\limsup_{\rho\searrow0}\frac{\log N(E,\rho)}{\log(1/\rho)}
\]
By (\ref{eq:epsilon-n-bound}), $\log\varepsilon_{n}/\log\varepsilon_{n+1}\rightarrow1$,
so one can compute the limit above along $\rho=\varepsilon_{n}$.
Also, $d(E_{n},E)=O(\varepsilon_{n})$ by Lemma (\ref{lem:En-properties}),
which implies $\log N(E,\varepsilon_{n})=\log N(E_{n},\varepsilon_{n})+O(1)$.
Thus, 
\[
\overline{\dim}_{B}E=\limsup_{n\rightarrow\infty}\frac{N(E_{n},\varepsilon_{n})}{\log(1/\varepsilon_{n})}
\]
We can bound $N(E_{n},\varepsilon_{n})$ by the carnality $|E_{n}|$,
and a simple induction shows that $|E_{n}|\leq\prod_{k\leq n}(M_{k}+N_{k}+2)$.
We get
\begin{align*}
\frac{\log N(E_{n},\varepsilon_{n})}{\log(1/\varepsilon_{n})} & \leq\frac{\log\left(\prod_{k\leq n}(M_{k}+N_{k}+2)\right)}{\log\left(\prod_{k\leq n}(M_{k}N_{k})\right)}\leq\frac{\sum_{k\leq n}(2+o(1))k^{2}}{\sum_{k\leq n}(4+o(1))k^{2}}=\frac{1}{2}+o(1)
\end{align*}
Plugging this into the previous equation, we have $\overline{\dim}_{B}E\leq1/2$.

\section{\label{subsec:deleting-zero-density}Zero-density subsequences of
$\alpha\beta$-orbits }

In this section we prove Proposition \ref{prop:full-density-has-box-dim--half}.
If $1,\alpha,\beta$ are $\mathbb{Q}$-independent, then after deleting
a zero-density set from an $\alpha\beta$-orbit, the box dimension
remains at least $1/2$.

Let $W\in\{x,y\}^{\mathbb{N}}$ be an infinite word, and let $\alpha,\beta,1$
be linearly independent over $\mathbb{Q}$. Let $J\subseteq\mathbb{N}$
and recall that $W[\alpha,\beta]|_{J}$ denotes the restriction of
the $\alpha\beta$-orbit $W[\alpha,\beta]$ to times in $J$ (i.e.,
deleting the times in $\mathbb{N}\setminus J$).

Suppose by way of contradiction that $\overline{d}(J)=1$ and $\delta=\underline{\dim}_{B}W[\alpha,\beta]|_{J}<1/2$. 

Then there exists a sequence $\delta_{n}\rightarrow\delta$ and an
unbounded subset $I\subseteq\mathbb{N}$ such that 
\[
N(W[\alpha,\beta]|_{J},2^{-n})\leq2^{n\delta_{n}}\qquad\text{for }n\in I
\]
Since $\overline{d}(J)=1$, for each $\ell\in\mathbb{N}$ there exists
an integer interval $I_{\ell}\subseteq J$ of length $\ell$. Set
$W_{\ell}=W|_{I_{\ell}}\in\{x,y\}^{\ell}$. Then $W_{\ell}[\alpha,\beta]$
is, up to a translate, contained in $W[\alpha,\beta]|_{J}$, so 
\[
N(W_{\ell}[\alpha,\beta],2^{-n})\leq N(W[\alpha,\beta]|_{J},2^{-n})
\]
Passing to a subsequence such that $W_{\ell}\rightarrow W'\in\{x,y\}^{\mathbb{N}}$
in the natural sense, every initial segment of $W'[\alpha,\beta]$
is an initial segment of $W_{\ell}[\alpha,\beta]$ for some $\ell$,
and therefore 
\[
N(W'[\alpha,\beta],2^{-n})\leq2^{n\delta_{n}}\qquad\text{for }n\in I
\]
so $\underline{\dim}_{B}(W'[\alpha,\beta])=\delta<1/2$, which contradicts
the theorem of Feng-Xiong.
\begin{rem}
The argument above shows that if $J\subseteq\mathbb{N}$ contains
arbitrarily long intervals (e.g. is the complement of a density zero
set) then 
\[
\dim_{B}W[\alpha,\beta]|_{n\in J}\geq\inf\{\text{dimension of an }\alpha\beta\text{-orbit\}}
\]
(the infimum being over $\alpha\beta$-orbits for the given $\alpha,\beta$).
\end{rem}

\section{\label{sec:Zero-Bdim-after-positive-density-deletion}Zero box dimension
on positive density sequences }

In this section we prove Theorem \ref{thm:full-density-box-dim-zero}
by constructing $\alpha,\beta\in\mathbb{R}$ such that $1,\alpha,\beta$
are rationally independent over $\mathbb{Q}$, an $\alpha\beta$-orbit
$(t_{n})_{n=1}^{\infty}$, and sets $I_{n}\subseteq\mathbb{N}$ with
$\underline{d}(I_{n})$$\rightarrow1$, such that $\dim_{B}\{t_{k}\}_{k\in I_{n}}=0$
for all $n$. Equivalently, we seek $J_{n}\subseteq\mathbb{N}$ with
$\overline{d}(J_{n})\rightarrow0$ such that $\dim_{B}\{t_{k}\}_{k\in\mathbb{N}\setminus J_{n}}=0$
for all $n$.

\subsection{Setup}

The construction is recursive. In order to keep track of errors at
stage $n$, we introduce the terms
\[
\sigma_{0}(n),\sigma_{1}(n),\sigma_{2}(n),\ldots
\]
These will always satisfy $0\leq\sigma_{k}(n)<1/(100n^{2})$, but
we leave it to the reader to verify that the accumulated errors do
not exceed this target at any stage. 

\subsection{The construction}

We shall define the following objects:
\begin{itemize}
\item $\alpha_{n},\beta_{n}>0$ .
\item $\varepsilon_{n}>0$, rapidly decreasing to zero.

We will be free to choose $\varepsilon_{n}$ arbitrarily at each step,
and for $n\geq2$ choose 
\begin{equation}
\varepsilon_{n}=\varepsilon_{n-1}^{1000n^{3}}\label{eq:epsilon-n-decay}
\end{equation}

\item An increasing sequence of words 
\[
W_{n}\in\{x,y\}^{*}
\]
For $n>1$ the words $W_{n}$ will satisfy
\[
W_{n}(\alpha_{n},\beta_{n})=\varepsilon_{n}
\]
their length will be 
\[
N_{n}=|W_{n}|=\varepsilon_{n-1}^{-1/(n-1)+\sigma_{0}(n)}
\]
Also, the symbol counts will satisfy
\begin{equation}
\frac{n+1}{2n+1}<\frac{|W_{n}|_{x}}{|W_{n}|_{y}}<\frac{2n+1}{n+1}\label{eq:symbol-ballance}
\end{equation}
and in particular $1/2<|W_{n}|_{x}/|W_{n}|_{y}<2$. 
\end{itemize}
Further properties will emerge in the course of the construction.
Eventually we will take $\alpha=\lim\alpha_{n}$, $\beta=\lim\beta_{n}$
and $W=\lim W_{n}$, and these will define the orbit. 

Begin with $\alpha_{1}=\beta_{1}=1/2$, let $N_{1}=2000$ and $W_{1}=x^{1000}y^{1000}$,
and $\varepsilon_{1}=10^{-1000}$. 

Now suppose we have defined carried out the construction up to and
including step $n\geq1$. Denote the letter counts in $W_{n}$ by
\[
k_{n}=|W_{n}|_{x}\quad,\quad\ell_{n}=|W_{n}|_{y}
\]
Choose an integer $L_{n}$ satisfying
\[
(L_{n}+3L_{n}^{1/2})N_{n}=\varepsilon_{n}^{-1/n+\sigma_{1}(n)}
\]
Since $N_{n}=\varepsilon_{n-1}^{-1/(n-1)+\sigma_{0}(n)}=\varepsilon_{n}^{-1/(1000n^{4})+\sigma_{2}(n)}$,
we have 
\[
L_{n}=\varepsilon_{n}^{-1/n+\sigma_{3}(n)}
\]

Next, define the word $V_{n}\in\{x,y\}^{*}$ by
\[
V_{n}=(x^{(2k_{n}-\ell_{n})}y^{(2\ell_{n}+k_{n})})^{L_{n}^{1/2}}
\]
This word was chosen so that the symbol counts in $V_{n}$ satisfy
\[
(|V_{n}|_{x},|V_{n}|_{y})=L_{n}^{1/2}\left(2(k_{n},\ell_{n})+(k_{n},\ell_{n})^{\perp}\right)
\]
where we write $(p,q)^{\perp}=(-q,p)$. Note that $2k_{n}-\ell_{n},2\ell_{n}-k_{n}>0$
by (\ref{eq:symbol-ballance}).

Define $W_{n+1}$ as the concatenation
\[
W_{n+1}=(W_{n})^{L_{n}}V_{n}
\]
The length of this word is 
\begin{align*}
N_{n+1} & =(L_{n}+O(L_{n}^{1/2}))N_{n}\\
 & =\varepsilon_{n}^{-1/n+\sigma_{4}(n)}\cdot\varepsilon_{n-1}^{-1/(n-1)+\sigma_{0}(n)}\\
 & =\varepsilon_{n}^{-1/n+\sigma_{5}(n)}
\end{align*}
where at the end we used (\ref{eq:epsilon-n-decay}) again. It is
not hard to check that (\ref{eq:symbol-ballance}) holds for the symbol
counts in $W_{n+1}$, because although ther frequencies in $V_{n}$
can deviate substantially from $1/2$, the $L_{n}$-fold repetition
of $W_{n}$ drowns out this effect in $W_{n+1}$.

Introduce a parameter $t>0$, and set 
\[
(\alpha_{n+1},\beta_{n+1})=(\alpha_{n},\beta_{n})-t\cdot(k_{n},\ell_{n})^{\perp}
\]
Note that 
\begin{equation}
W_{n}(\alpha_{n+1},\beta_{n+1})=W_{n}(\alpha_{n},\beta_{n})=\varepsilon_{n}\label{eq:alphabeta-update-preserves-wn}
\end{equation}
because the change to the vector $(\alpha_{n},\beta_{n})$ is orthogonal
to the symbol counts of $W_{n}$. Note that this property does not
persist more than one step (i.e. in general $W_{n}(\alpha_{n+2},\beta_{n+2})\neq W_{n}(\alpha_{n},\beta_{n})$).

Next we evaluate $W_{n+1}$ on $\alpha_{n+1},\beta_{n+1}$. First
note that
\[
V_{n}(\alpha_{n+1},\beta_{n+1})=O\left(L_{n}^{1/2}\cdot\left(W_{n}(\alpha_{n},\beta_{n})-t(k_{n}^{2}+\ell_{n}^{2})\right)\right)
\]
so
\[
W_{n+1}(\alpha_{n+1},\beta_{n+1})=(L_{n}+O(L_{n}^{1/2}))\cdot W_{n}(\alpha_{n},\beta_{n})-t\cdot O(L_{n}^{1/2}(k_{n}^{2}+\ell_{n}^{2}))
\]
We first choose $t$ so as to make this equal to zero: since
\begin{align*}
(L_{n}+O(L_{n}^{1/2}))\cdot W_{n}(\alpha_{n},\beta_{n}) & =\varepsilon_{n}^{-1/n+\sigma_{4}(n)}\cdot\varepsilon_{n}\\
 & =\varepsilon_{n}^{1-1/n+\sigma_{4}(n)}
\end{align*}
the solution $t$ of the equation $W_{n+1}(\alpha_{n+1},\beta_{n+1})=0$
is of order $\varepsilon_{n}^{1-1/n+\sigma_{3}(n)}/(L_{n}^{1/2}(k_{n}^{2}+\ell_{n}^{2}))$,
and for this choice of $t$ the change to $\alpha_{n},\beta_{n}$
is bounded by 
\begin{equation}
\left\Vert (\alpha_{n+1},\beta_{n+1})-(\alpha_{n},\beta_{n})\right\Vert =t\left\Vert (k_{n},\ell_{n})^{\perp}\right\Vert =O(\frac{1}{L_{n}^{1/2}}\varepsilon_{n}^{1-1/n+\sigma_{4}(n)})<\varepsilon_{n}^{1/2}\label{eq:size-of-change-to-alphan-betan}
\end{equation}
Furthermore, by making an additional perturbation of $t$ of a much
smaller order, we can ensure that the number
\[
\varepsilon_{n+1}=W_{n+1}(\alpha_{n+1},\beta_{n+1})
\]
satisfies (\ref{eq:epsilon-n-decay}), and we still have (\ref{eq:alphabeta-update-preserves-wn})
and (\ref{eq:size-of-change-to-alphan-betan}).

\subsection{Analysis of the example}

By construction, $W_{n}$ is an initial segment of $W_{n+1}$, so
\[
W=\lim W_{n}
\]
is well defined. Also, the rapid decay of $\varepsilon_{n}$, and
the fact that $\alpha_{n},\beta_{n}$ change by at most $\varepsilon_{n}^{1/2}$
at each step, mean that $(\alpha_{n}),(\beta_{n})$ are Cauchy sequences,
and we can define 
\[
\alpha=\lim\alpha_{n}\quad,\quad\beta=\lim\beta_{n}
\]
Since $\alpha_{n}\rightarrow\alpha$, $\beta\rightarrow\beta$ clearly
$W[\alpha_{n},\beta_{n}]_{i}\rightarrow W[\alpha,\beta]_{i}$ for
all $i\in\mathbb{N}$. Let us estimate the rate of convergence more
precisely. For $m\in\mathbb{N}$, write 
\[
\delta_{m}=\{\max|\alpha_{m+1}-\alpha_{m}|,|\beta_{m+1}-\beta_{m}|\}<\varepsilon_{m}^{1/2}
\]
By construction, for any $n\leq m$ we have the bound
\[
|W_{n}(\alpha_{m+1},\beta_{m+1})-W_{n}(\alpha_{m},\beta_{m})|<N_{n}\delta_{m}<N_{n}\varepsilon_{m}^{1/2}<\varepsilon_{m}^{1/3}
\]
Using $W_{n}(\alpha_{n+1},\beta_{n+1})=W_{n}(\alpha_{n},\beta_{n})=\varepsilon_{n}$
(which is (\ref{eq:alphabeta-update-preserves-wn})), it follows that
\[
|W(\alpha,\beta)_{N_{n}}-\varepsilon_{n}|\leq\sum_{m=n+1}^{\infty}\varepsilon_{m}^{1/3}<2\varepsilon_{n}
\]

Let $J_{n}\subseteq\mathbb{N}$ denote the indices that lie in copie
of $V_{n}$ in $W$. This includes the blocks at the end of the $W_{n}$
itself that was placed there explicitly in the construction, but also
copies of $V_{n}$ in $W_{n+1}$ in the long initial segment $(W_{n})^{L_{n}}$
of $W_{n+1}$, and these repeat again in $W_{n'}$ for all $n'>n$.
It is easy to see that $J_{n}$ has positive density, but $\overline{d}(J_{n})\rightarrow0$.
But the densities decay rapidly enough that the set 
\[
J_{>n}=\bigcup_{i=n+1}^{\infty}J_{i}
\]
also satisfies $\overline{d}(J_{>n})\rightarrow0$ as well. 

We aim to show that
\[
\forall n\in\mathbb{N}\quad\dim_{B}W[\alpha,\beta]|_{\mathbb{N}\setminus J_{>n}}=0
\]
Fix $n_{0}$ and let $j\in\mathbb{N}\setminus J_{>n_{0}}$. Let $n_{1}$
be the least index such that $j\leq N_{n_{1}}$, or $n_{0}$, whichever
is greater. Using the recursive structure of $W$ it is easy to see
that there is a splitting
\[
W|_{[1,j]}=W'W''
\]
with
\begin{itemize}
\item $W'=W_{n_{1}}^{m(n_{1})}W_{n_{1}-1}^{m(n_{1}-1)}\ldots W_{n_{0}}^{m(n_{0})}$
for integers $m(n_{1}),m(n_{1}-1)\ldots,m(n_{0})$ satisfying $0\leq m(n)\leq L_{n-1}$
.
\item $W''$ is an initial segment of $W_{n_{0}}$. 
\end{itemize}
Now, clearly
\begin{align*}
W'(\alpha,\beta) & <\sum_{n=n_{0}}^{n_{1}}m(n)2\varepsilon_{n}\\
 & \leq2\sum_{n=n_{0}}^{n_{1}}L_{n-1}\varepsilon_{n}\\
 & <2\sum_{n=n_{0}}^{\infty}\varepsilon_{n-1}^{-1/n+\sigma_{3}(n)}\varepsilon_{n}\\
 & <\varepsilon_{n_{0}}^{1/2}
\end{align*}
and it follows that $W(\alpha,\beta)_{j}$ is within $\varepsilon_{n_{0}}^{1/2}$
of $W''(\alpha,\beta)_{j-|W'|}$. Thus, we have shown that
\begin{align*}
N(W[\alpha,\beta]|_{\mathbb{N}\setminus J_{>n_{0}}},2\varepsilon_{n_{0}}^{1/2}) & \leq N(W_{n_{0}}[\alpha,\beta],\varepsilon_{n_{0}}^{1/2})\\
 & \leq N_{n_{0}}
\end{align*}
For $n\geq n_{0}$, since $J_{>n}\subseteq J_{>n_{0}}$ we have $\mathbb{N}\setminus J_{n_{0}}\subseteq\mathbb{N}\setminus J_{n}$,
hence
\[
N(W[\alpha,\beta]|_{\mathbb{N}\setminus J_{>n_{0}}},2\varepsilon_{n}^{1/2})\leq N(W[\alpha,\beta]|_{\mathbb{N}\setminus J_{>n}},2\varepsilon_{n}^{1/2})\leq N_{n}
\]
For any $2\varepsilon_{n}^{1/2}<\delta\leq2\varepsilon_{n-1}^{1/2}$
we have
\begin{align*}
N(W[\alpha,\beta]|_{\mathbb{N}\setminus J_{>n_{0}}},\delta) & \leq N(W[\alpha,\beta]|_{\mathbb{N}\setminus J_{>n_{0}}},2\varepsilon_{n}^{1/2})\\
 & \leq N_{n}\\
 & =\varepsilon_{n-1}^{-1/(n-1)+\sigma_{0}(n-1)}\\
 & =\delta^{-o(1)}
\end{align*}
This holds for any $0<\delta<2\varepsilon_{n_{0}}^{1/2}$ because
there is always some $n\geq n_{0}$ such that $2\varepsilon_{n}^{1/2}<\delta\leq2\varepsilon_{n-1}^{1/2}$,
so we have shown that 
\[
\dim_{B}W[\alpha,\beta]|_{\mathbb{N}\setminus J_{>n_{0}}}=0
\]
which is what we were out to prove.

\section{\label{sec:Positive-Assouad-dimension-after-positive-density-eletion}Survival
of positive Assouad dimension }

In this section we prove Theorem \ref{thm:Assouad-along-subsequence},
that is, that if $1,\alpha,\beta$ are $\mathbb{Q}$-independent,
then any positive-upper-density subsequence of an $\alpha\beta$-orbit
has Assouad dimension at least $1/4$.

\subsection{Some multilinear Diophantine approximation}

In this section we assume that $\alpha,\beta,1$ are independent over
$\mathbb{Q}$. 

Let $v^{1},v^{2},v^{3}\in(\mathbb{N}\cup\{0\})^{2}$ with $v^{i}=(v_{1}^{i},v_{2}^{i})$
and write 
\[
k_{i}=v_{1}^{i}+v_{2}^{i}
\]

Let $\left\Vert t\right\Vert $ denote the distance of $t\in\mathbb{R}$
to the nearest integer and write
\[
\varepsilon_{i}=\left\Vert v^{i}(\alpha,\beta)^{T}\right\Vert =\left\Vert v_{1}^{i}\alpha+v_{2}^{i}\beta\right\Vert 
\]
so $\varepsilon_{i}\in[0,1)$ and there are integers $m_{i}$ such
that
\[
v_{i}\cdot\left(\begin{array}{c}
\alpha\\
\beta
\end{array}\right)=m_{i}+\varepsilon_{i}
\]

\begin{lem}
If $\varepsilon_{j}/\varepsilon_{i}\in\mathbb{N}$ then $k_{i}|k_{j}$
and $v^{j}=\frac{k_{j}}{k_{i}}v_{i}$.
\end{lem}
\begin{proof}
Suppose $\ell=\varepsilon_{j}/\varepsilon_{i}\in\mathbb{N}$. Then
$v^{j}(\alpha,\beta)^{T}-m_{j}=\ell\cdot(v_{i}(\alpha,\beta)^{T}-m_{i})$,
which is equivalent to 
\[
(v_{1}^{j}-\ell v_{1}^{i})\alpha+(v_{2}^{j}-\ell v_{2}^{i})\beta=m_{j}-\ell m_{i}
\]
By independence, the coefficients of $\alpha,\beta$ are zero, and
so is the right hand side. Thus $v^{j}=\ell v^{i}$ and (since the
coordinates of $v^{i},v^{j}$ sum to $k_{i},k_{j}$ respectively)
we have $k_{j}=\ell k_{i}$. Substituting $\ell=k_{j}/k_{i}$ into
$v^{j}=\ell v^{i}$ completes the proof. 
\end{proof}
\begin{lem}
\label{lem:colinearity-of-integer-vectors-implies-integer-colinearity}If
$u,v\in\mathbb{Z}^{2}$ are co-linear over $\mathbb{R}$ then $u,v$
are integer multiples of an integer vector $w\in\mathbb{Z}^{2}$.
\end{lem}

\begin{lem}
\label{lem:no-close-minima}Let $0<s<1/2$ and $t>1+2s$. In the notation
above, suppose that $\varepsilon_{1}$ is sufficiently small relative
to $s,t$, that $k_{1}<k_{2}<k_{3}\leq1/\varepsilon_{1}^{s}$, that
$\varepsilon_{2}<\varepsilon_{1}^{t}$ and that $\varepsilon_{3}<\varepsilon_{1}^{t}$.
Then $k_{2}|k_{3}$ and $\varepsilon_{3}=\frac{k_{3}}{k_{2}}\varepsilon_{2}$.
\end{lem}
\begin{proof}
Let $A$ denote the matrix whose rows are $v^{2}$ and $v^{3}$. We
claim that $A$ is singular. For suppose otherwise. Then Since
\[
A\left(\begin{array}{c}
\alpha\\
\beta
\end{array}\right)=\left(\begin{array}{c}
m_{1}\\
m_{2}
\end{array}\right)+\left(\begin{array}{c}
\varepsilon_{2}\\
\varepsilon_{2}
\end{array}\right)
\]
Since $A$ is assumed to be invertible, 
\begin{align*}
m_{1}+\varepsilon_{1} & =v^{1}\left(\begin{array}{c}
\alpha\\
\beta
\end{array}\right)\\
 & =v^{1}A^{-1}A\left(\begin{array}{c}
\alpha\\
\beta
\end{array}\right)\\
 & =\frac{1}{\det A}v^{1}\left(\begin{array}{cc}
v_{2}^{3} & -v_{2}^{2}\\
-v_{1}^{3} & v_{1}^{2}
\end{array}\right)\left(\left(\begin{array}{c}
m_{2}\\
m_{3}
\end{array}\right)+\left(\begin{array}{c}
\varepsilon_{1}\\
\varepsilon_{2}
\end{array}\right)\right)
\end{align*}
Write $d=\det A$. Multiplying by $d$ and using the fact that $d\in\mathbb{N}$
and all matrices and vectors above other than $(\varepsilon_{2},\varepsilon_{3})$
have integer entries, 
\begin{align*}
d\varepsilon_{1} & \in\mathbb{\mathbb{Z}}+v^{1}\left(\begin{array}{cc}
v_{2}^{3} & -v_{2}^{2}\\
-v_{1}^{3} & v_{1}^{2}
\end{array}\right)\left(\begin{array}{c}
\varepsilon_{2}\\
\varepsilon_{3}
\end{array}\right)
\end{align*}
The vector $v^{1}$ and matrix on the right hand side have entries
that are $<1/\varepsilon_{1}^{s}$ so 
\[
d\varepsilon_{1}\in\mathbb{Z}+O((1/\varepsilon_{1}^{s})^{2}\varepsilon_{1}^{t})=\mathbb{Z}+O(\varepsilon_{1}^{t-2s})
\]
We have also assumed that $t-2s>1$, hence $\varepsilon_{1}^{t-2s}\ll\varepsilon_{1}$,
and the equation above shows that $d\varepsilon_{1}$ is closer than
$\varepsilon_{1}$ to an integer. On the other hand also $1\leq d\leq1/\varepsilon_{1}^{2s}$
so $\varepsilon_{1}\leq d\varepsilon_{1}\leq\varepsilon_{1}^{1-2s}$.
Since we have assumed that $1-2s>0$, if $\varepsilon_{1}$ is small
enough enough in a manner depending only on $s,t$, then the interval
$[\varepsilon_{1},\varepsilon_{1}^{t-2s}]$ is at least $\varepsilon_{1}$
far from every integer. This is a contradiction.

Thus, $A$ is singular, so $v^{3}$is co linear with $v^{2}$, and
by Lemma \ref{lem:colinearity-of-integer-vectors-implies-integer-colinearity},
$n_{2}|n_{3}$ and $v_{3}=\frac{n_{3}}{n_{2}}v_{2}$. This implies
that $\varepsilon_{3}=\left\Vert \frac{n_{3}}{n_{2}}\varepsilon_{2}\right\Vert $.
But $n_{3}/n_{2}\leq n_{3}\leq1/\varepsilon_{1}^{s}$ and since $s<t$
we have $\frac{n_{3}}{n_{2}}\varepsilon_{2}<\varepsilon_{1}^{-s}\varepsilon_{1}^{t}<1/2$
if $\varepsilon_{1}$ is large enough. This implies $\varepsilon_{3}=\frac{n_{3}}{n_{2}}\varepsilon_{2}$
(without taking the distance to the nearest integer) as desired. 
\end{proof}

\subsection{Covering numbers of $\alpha\beta$-orbit}

Let us now define 
\[
\delta_{n}=\min\{\left\Vert a\alpha+b\beta\right\Vert \mid0\leq a,b\in\mathbb{Z}\;,\;a+b=n\}
\]
We say that $\delta_{n}$ is minimal if $\delta_{m}\geq\delta_{n}$
for all $1\leq m<n$. 

Let $u_{n}\in(\mathbb{N}\cup\{0\})^{2}$ be the vector realizing the
minimum in the definition of $\delta_{n}$ and note that if $(t_{n})$
is an $\alpha\beta$-orbit then 
\begin{equation}
d(t_{i},t_{j})\geq\delta_{|i-j|}\label{eq:distance-in-terms-of-delta-n}
\end{equation}
Also notice that for any integers $n_{1}<n_{2}<n_{3}$ we can apply
the previous lemma to $v^{i}=u^{n_{i}}$, and then $k_{i}=n_{i}$
and $\varepsilon_{i}=\delta_{n_{i}}$. 
\begin{lem}
Let $0<s<1/2$ and $t>1+2s$. Let $(t_{n})$ be an $\alpha\beta$-orbit,
let $n$ be sufficiently large in a manner depending on $\alpha,\beta,s,t$,
and let $\delta_{n}$ be minimal. Set $N=\left\lceil 1/\delta_{n}^{s}\right\rceil $
and suppose there exists $m\leq N$ with $\delta_{m}$ minimal and
$\delta_{m}<\delta_{n}^{t}$. Then
\begin{enumerate}
\item $t_{1}\ldots t_{N}$ are $\delta_{m}$-separated.
\item For every $1\leq i<j\leq N$, either $d(t_{i},t_{j})\geq\delta_{n}^{t}$
or $d(t_{i},t_{j})\leq\delta_{m}/\delta_{n}^{s}$.
\end{enumerate}
\end{lem}
\begin{proof}
(1) By minimality of $\delta_{n}$ we have $\delta_{k}\geq\delta_{n}$
for $1\leq k\leq n$, and from Lemma \ref{lem:no-close-minima} we
conclude that $\delta_{k}\geq\delta_{m}$ for $n<k\leq N$. The $\delta_{m}$-separation
of $t_{1}\ldots,t_{N}$ follows from (\ref{eq:distance-in-terms-of-delta-n}).

For (2) apply the lemma again to get that if $1\leq k\leq N$ and
$\delta_{k}<\delta_{n}^{t}$ then $m|k$ and 
\[
\delta_{k}=\frac{k}{m}\delta_{m}\leq N\delta_{m}\leq\delta_{n}^{-s}\delta_{m}
\]
Now let $1\leq i<j\leq N$, write $k=j-i$ and suppose $d(t_{i},t_{j})<\delta_{n}^{t}$.
By (\ref{eq:distance-in-terms-of-delta-n}) and $t>1$ this means
$\delta_{k}<\delta_{n}$, so by minimality of $n$ we have $n<k\leq N$.
Apply Lemma \ref{lem:no-close-minima} to $u_{n},u_{m}$ and $u'=u_{j}-u_{i}$,
and conclude that $d(t_{i},t_{j})=\frac{k}{m}\delta_{m}\leq\delta_{m}/\delta_{n}^{s}$.
\end{proof}
\begin{quote}
\end{quote}

\subsection{\label{subsec:Proof-of-Assouad-result}Proof of Theorem \ref{thm:Assouad-along-subsequence}}

Let $\alpha,\beta,1$ be independent over $\mathbb{Q}$.

Let $\{t_{i}\}_{i=1}^{\infty}$ be an $\alpha\beta$-orbit $(t_{n})_{n=1}^{\infty}$.

Let $U\subseteq\mathbb{N}$ with $\rho=\overline{d}(U)>0$, and set
$\rho_{N}=|U\cap\{1,\ldots,N\}|/N$, so $\rho_{N}=\rho+o(1)$.

We wish to prove Theorem \ref{thm:full-density-box-dim-zero}, i.e.,
that $\dim_{A}\{t_{i}\}_{i\in U}\geq1/4$, where for $E\subseteq\mathbb{R}$,
the Assouad dimension $\dim_{A}E$ of $E$ is defined by 
\[
\dim_{A}E=\limsup_{\delta\rightarrow0}\;\sup\{\frac{\log N(E\cap J,\delta|J|)}{\log(1/\delta)}\mid J\subseteq\mathbb{R}\text{ an interval}\}
\]
We note the trivial bound $\dim_{A}E\geq\overline{\dim}_{B}E$ .

Introduce parameters $0<s<1/2$ and $t>1-2s$ which we shall optimize
later.

Fix a large $n$ with $\delta_{n}$ minimal, and set $N=N_{n}=\left\lceil 1/\delta_{n}^{s}\right\rceil $.
\begin{enumerate}
\item If $\delta_{m}>\delta_{n}^{t}$ for all $n\leq m\leq N$ then, since
by minimality of $n$ also $\delta_{m}>\delta_{n}>\delta_{n}^{t}$
for $1\leq m\leq n$, we have 
\[
N(\{t_{i}\}_{i\in U},\delta_{n}^{t})\geq|U\cap\{1,\ldots,N\}|=\rho_{N}\cdot N=(\rho+o(1))\delta_{n}^{-s}=(\delta_{n}^{t})^{-s/t+o(1)}
\]
If this holds for infinitely many $n$ we have $\overline{\dim}_{B}\{t_{i}\}_{i\in U}\geq s/t$.
\item Otherwise, let 
\[
E_{n}=\{t_{i}\mid i\in U\cap\{1,\ldots,N\}\}
\]
Choose a maximal $\delta_{n}^{t}/2$-separated subset $F_{n}\subseteq E_{n}$.
Then each of the $(\rho+o(1))N$ points in $E_{n}$ is within $\delta_{n}^{t}/2$
of a point in $F_{n}$.

Choose another parameter $r>0$. 
\begin{itemize}
\item If $|F_{n}|>\delta_{n}^{-r}$ then
\[
N(F_{n},\delta_{n}^{t})\geq\delta_{n}^{-r}=(\delta_{n}^{t})^{-r/t}
\]
and if this holds for infinitely many $n$ then we again get $\overline{\dim}_{B}\{t_{i}\}_{i\in U}\geq r/t$. 
\item Otherwise $|F_{n}|\leq\delta_{n}^{-r}$ for all sufficiently large
$n$. In this case, for some $y$ out of the $\leq$$\delta_{n}^{-r}$
points in $F_{n}$, the number of points from $E_{n}$ in the interval
$I=(y-\delta_{n}^{2}/2,y+\delta^{2}/2)$. must be at least 
\[
(\rho+o(1))N/\delta_{n}^{-r}=(\rho+o(1))\delta_{n}^{s-r}
\]
By the previous lemma, all these points are in fact contained in the
interval $J=(y-\delta_{m}/\delta_{n}^{s},y+\delta_{m}/\delta_{n}^{s})$.
Thus
\[
N(E_{n}\cap J,\delta_{m})\geq(\rho+o(1))\delta_{n}^{s-r}=(\frac{\delta_{m}}{|J|})^{-r+o(1)}
\]
and if this happens for infinitely many $n$ we get $\dim_{A}\{t_{i}\}_{i\in U}\geq r$.
\end{itemize}
\end{enumerate}
Taking stock, we have shown that $\dim_{A}W[\alpha,\beta]|_{U}$ is
at least $\min\{s/t,r/t,r\}$. Optimizing this quantity over the parameter
ranges $0<s<1/2$ , $t>1+2s$ and $0<r<1$, we get a lower bound of
$1/4$ when $r\approx1/2$, $t\approx2$ and $r=1/2$. 

\bibliographystyle{plain}
\bibliography{bib}

\begin{thebibliography}{1}

\bibitem{AlgomHochmanWu2024}
Amir Algom, Michael Hochman, and Meng Wu.
\newblock On {F}eng's embedding problem.
\newblock {\em preprint}, 2025.

\bibitem{Baker2021}
Simon Baker.
\newblock New dimension bounds for {$\alpha\beta$} sets.
\newblock {\em J. Number Theory}, 228:59--72, 2021.

\bibitem{Engelking1961}
R.~Engelking.
\newblock Sur un probl\`eme de {K}. {U}rbanik concernant les ensembles
  lin\'{e}aires.
\newblock {\em Colloq. Math.}, 8:243--250, 1961.

\bibitem{FengXiong2018}
De-Jun Feng and Ying Xiong.
\newblock Affine embeddings of {C}antor sets and dimension of
  {$\alpha\beta$}-sets.
\newblock {\em Israel J. Math.}, 226(2):805--826, 2018.

\bibitem{Katznelson1979}
Y.~Katznelson.
\newblock On {$\alpha \beta $}-sets.
\newblock {\em Israel J. Math.}, 33(1):1--4, 1979.

\bibitem{Shmerkin2019}
Pablo Shmerkin.
\newblock On {F}urstenberg's intersection conjecture, self-similar measures,
  and the {$L^q$} norms of convolutions.
\newblock {\em Ann. of Math. (2)}, 189(2):319--391, 2019.

\bibitem{Wu2019}
Meng Wu.
\newblock A proof of {F}urstenberg's conjecture on the intersections of
  {$\times p$}- and {$\times q$}-invariant sets.
\newblock {\em Ann. of Math. (2)}, 189(3):707--751, 2019.

\bibitem{Yu2019}
Han Yu.
\newblock Multi-rotations on the unit circle.
\newblock {\em J. Number Theory}, 200:316--328, 2019.

\end{thebibliography}

\bigskip
\bigskip
\footnotesize
\noindent{}\texttt{Department of Mathematics, The Hebrew University of Jerusalem, Jerusalem, Israel\\ email: michael.hochman@mail.huji.ac.il}
\end{document}